\documentclass[12pt]{article}
\usepackage{amsmath,amsfonts,amssymb,cite}
\usepackage{graphicx}
\usepackage{amsmath,amsfonts,ifthen,fullpage}  %,amsthm}

\title{\bf On the number of harmonic frames}

\author{Simon Marshall$^*$ and Shayne Waldron$^\dagger$ 
 \\   \\
\vbox{
\hbox{\footnotesize \noindent $^*$ Department of Mathematics, University of Wisconsin--Madison, 480 Lincoln Drive, Madison, WI 53706, USA}
{\vskip 0.10 truecm}
\hbox{\footnotesize ${}^\dagger$ Department of Mathematics, University of Auckland,
Private Bag 92019, Auckland, New Zealand}
{\vskip 0.10 truecm}
} 
}

\date{\today}

\newtheorem{theorem}{Theorem}[section]
\newtheorem{lemma}[theorem]{Lemma}
\newtheorem{example}[theorem]{Example}

\newtheorem{proposition}[theorem]{Proposition}

\newenvironment{proof}{{\noindent \bf
Proof:}}{\hfill$\Box$\bigskip}

\newcommand{\CC}{\mathbb{C}}
\newcommand{\RR}{\mathbb{R}}

\newcommand{\ZZ}{\mathbb{Z}}
\newcommand{\C}{\mathbb{C}}

\newcommand{\Q}{\mathbb{Q}}

\def\Cd{\CC^d}
\def\Rd{\RR^d}

\def\bmat#1{\begin{bmatrix}#1\end{bmatrix}}
\def\inpro#1{\langle#1\rangle}
\def\norm#1{\Vert#1\Vert}%\def\norm#1{\|#1\|}

\let\ga\alpha
\let\gb\beta
\let\go\omega

\let\gep\epsilon

\let\gs\sigma

\def\ord{\mathop{\rm ord}\nolimits}

\def\Iff{\hskip1em\Longleftrightarrow\hskip1em}
\def\TT{\mathbb{T}}

\def\Aut{\mathop{\rm Aut}\nolimits}
\def\Fix{\mathop{\rm Fix}\nolimits}

\def\cX{\mathcal{X}}
\def\cY{\mathcal{Y}}

\newif\ifdraft\def\draft{\drafttrue\hoffset=.8truecm\showlabeltrue
\def\comment##1{{\bf comment: ##1}}
\headline={\sevenrm \hfill \ifx\filenamed\undefined\jobname\else\filenamed\fi%
(.tex) (as of \ifx\updated\undefined???\else\updated\fi)
 \TeX'ed at {\hour\time\divide\hour by 60{}%
\minutes\hour\multiply\minutes by 60{}%
\advance\time by -\minutes
\the\hour:\ifnum\time<10{}0\fi\the\time\  on \today\hfill}}
}

\begin{document}

\maketitle
\bigskip

\begin{abstract}

There is a finite number $h_{n,d}$
of tight frames of $n$ distinct vectors
for $\Cd$ which are the orbit of a %single 
vector under a unitary action of
the cyclic group $\ZZ_n$.
%an abelian group $G$.
%(up to unitary equivalence).
% (for nonabelian groups there are uncountably many). 
These {\it cyclic harmonic frames} (or {\it geometrically uniform tight frames})
are used in signal analysis %processing 
and quantum information theory,
% (for $G$ the cyclic group), 
and provide many tight frames of particular interest.
Here we investigate the %long standing 
conjecture that
%the number $h_{n,d}$ of cyclic harmonic frames of $n$ vectors for $\Cd$ 
 $h_{n,d}$
grows like $n^{d-1}$.
By using a result of Laurent which describes
the set of solutions of algebraic equations in roots of unity,
we prove the asymptotic estimate
$$ h_{n,d} \approx {n^d \over \varphi(n)}\ge n^{d-1}, \qquad n\to\infty. $$
By using a group theoretic approach,
we also give some exact formulas for $h_{n,d}$, and estimate
the number of cyclic harmonic frames up to projective unitary equivalence.

% -- check that claim?
\end{abstract}

\bigskip
\vfill


\bigskip
\noindent {\bf AMS (MOS) Subject Classifications:}
primary
%42C05, \ifdraft	Orthogonal functions and polynomials, general theory \else\fi
42C15, \ifdraft General harmonic expansions, frames  \else\fi
94A12, \ifdraft	Signal theory (characterization, reconstruction, filtering, etc.) \else\fi

secondary
11Z05, \ifdraft	Miscellaneous applications of number theory \else\fi
20C15, \ifdraft	Ordinary representations and characters \else\fi
94A15, \ifdraft	Information theory, general [See also 62B10, 81P45] \else\fi
20F28, \ifdraft	Automorphism groups of groups \else\fi

%52B11, \ifdraft $n$-dimensional polytopes \else\fi
%52B15, \ifdraft	Symmetry properties of polytopes \else\fi

\vskip .5 truecm
\hrule
\newpage

\section{Introduction}

%Equal--norm Finite t
Tight frames of $n$ %distinct equal--norm 
vectors for $\Cd$
have numerous applications
(see the surveys \cite{KC07}, \cite{CK13}).
These include
signal transmission with erasures \cite{GKK01}, \cite{HP04}, \cite{BP05}
and quantum information theory \cite{RBSC04}, \cite{SG09}.

Many tight frames of practical and theoretical interest are
$G$--frames (the orbit of a unitary action of a group $G$) \cite{W13}.
Most notable are the {\it harmonic frames} ($G$ is abelian)
and {\it SICs}, i.e., $d^2$ equiangular lines in $\Cd$
(for a projective action of the abelian group $\ZZ_d^2$).
The main result of this paper  is a precise statement about how
numerous the harmonic frames of $n$ vectors for $\Cd$ are 
(Theorem \ref{framecount}). % (for $G$ cyclic). 
By way of comparison, SICs are known
to exist only for certain values of $d$, and there is
strong evidence for {\it Zauner's conjecture}
that they exist for all values of $d$ (see \cite{SG09}, \cite{Zau10}).

We now provide some background on harmonic frames,
and then detail our approach 
(precise definitions are given in \S\ref{harmframesbasics}).
%Precise definitions and results are given in \S\ref{harmframesbasics}.
What we will call a {\it cyclic harmonic frame} for $\Cd$
was first introduced as a % the 
$d\times n$ submatrix $[v_1,\ldots,v_n]$ of the Fourier matrix
(character table for $\ZZ_n$)
\begin{equation}
\label{cycordern}
\bmat{1&1&1&\cdots&1\\
1&\go&\go^2&\cdots&\go^{n-1}\\
1&\go^2&\go^4&\cdots&\go^{2(n-1)} \\
\vdots&\vdots&\vdots&&\vdots\\
1&\go^{n-1}&\go^{2(n-1)}&\cdots&\go^{(n-1)(n-1)} },
\qquad \go:=e^{2\pi i\over n},
\end{equation}
obtained by selecting $d$ of the rows (characters of $\ZZ_n$).
See \cite{GVT98},
\cite{HMRSU00},
\cite{CK03} (who use the term {\it harmonic frame} for when the
first $d$ rows are taken), and
\cite{CRT06} (who use the term a {\it Fourier ensemble}).
These tight frames can be viewed as $\ZZ_n$--frames
(called {\it geometrically uniform frames} in \cite{BE03}).
This construction generalises, with $\ZZ_n$ replaced
by an abelian group $G$ of order $n$ \cite{VW05},
to give what we call a {\it harmonic frame}
 (it is {\it cyclic} if $G$ can be taken to be $\ZZ_n$).
It follows from the character table construction
(and the fact there are a finite number of abelian groups of order $n$)
that there is a finite number of harmonic frames of $n$ vectors for $\Cd$.

A computer study \cite{HW06} of the harmonic frames of $n$ vectors for $\Cd$
suggested the following behaviour:
\begin{itemize}
\item The number of harmonic frames (up to unitary equivalence) grows
like $n^{d-1}$, and it is influenced by the prime factors of $n$.
\item The majority of harmonic frames are cyclic.
\end{itemize}
In this paper, we show that for fixed $d$ the number $h_{n,d}$
of cyclic harmonic frames grows like
$$ h_{n,d} \approx {n^d\over\varphi(n)} \ge n^{d-1}, \qquad n\to\infty. $$
The key points of our argument are
\begin{itemize}
\item Cyclic harmonic frames correspond to
$d$--element subsets $J\subset\ZZ_n$. % (by Pontryagin duality).
\item When cyclic harmonic frames given by $J,K\subset\ZZ_n$ are unitarily
equivalent, usually $K=\gs J$ for some automorphism. When this is not the
case, we say they % the frames 
are {\it exceptional}.
\item The automorphisms of $\ZZ_n$ are easy to describe (as the units $\ZZ_n^*$).
\item A pair of unitarily equivalent cyclic harmonic frames determine
a torsion point on the $(2d)$--torus $\TT^{2d}$.
\item By using results about the torsion point solutions of
algebraic equations, we show that the number of exceptional harmonic
frames grows slower than the number which aren't.
\item The nonexceptional cyclic harmonic frames are counted
by Burnside enumeration.
\end{itemize}

We carry out this argument in \S4--\S5.  We give examples and some numerical data in \S6.  In \S7 we show that there are no exceptional equivalences when $n$ is prime, and together with Burnside enumeration this allows us to give an exact formula for
$h_{n,d}$ in this case, which we break down into lifted and unlifted,
and into real and complex harmonic frames.

In the final section, % \S8, 
we use our techniques to investigate 
the number $p_{n,d}$ of harmonic frames of $n$ vectors for $\Cd$ 
up to {\it projective unitary equivalence}.
For $d\ge 4$, this gives the lower estimate
$$ p_{n,d} \approx {n^{d-1}\over \varphi(n)} \ge n^{d-2}, \qquad n\to\infty.
$$

\section{Harmonic frames}
\label{harmframesbasics}

A sequence of $n$ vectors $(v_j)$ in $\Cd$ is a {\bf tight frame}
for $\Cd$ if for some $A>0$
$$ A\norm{f}^2= \sum_{j=1}^n |\inpro{f,v_j}|^2, \qquad\forall f\in\Cd.$$
By the polarisation identity, this is equivalent to the
``redundant orthogonal expansion''
\begin{equation}
\label{redundonexp}
f={1\over A} \sum_{j=1}^n \inpro{f,v_j}v_j, \qquad\forall f\in\Cd.
\end{equation}
%A tight frame can be normalised so that the frame bound $A$ is $1$
%(take ${1\over\sqrt{A}v_j)$)kk
We say that tight frames $(v_j)$ and $(w_k)$ are 
{\bf unitarily equivalent} (up to a reindexing) if there
is a unitary map $U$ and a bijection $\gs:j\to k$ (a reindexing)
between their index sets for which
\begin{equation}
\label{unitaryequivdef}
v_j = U w_{\gs j}, \qquad\forall j.
\end{equation}
If $(v_j)$ is unitarily equivalent to a frame $(w_j)\subset\Rd$, then
we say it is a {\bf real frame}.

A tight frame $(gv)_{g\in G}$ which is the orbit of a vector $v$ under 
the unitary action of a finite group $G$ is called a {\bf $G$--frame}
(or {\bf group frame}) \cite{W13}. For $G$ abelian, there are finitely
many $G$--frames for $\Cd$ up to unitary equivalence, which we call
the {\bf harmonic frames}. 
For $G$ nonabelian,
there are uncountably many $G$--frames for $\Cd$, $d\ge2$.
We now give the basic theory of harmonic frames required
(see \cite{W16} for details).

Let $G$ be a finite abelian group.
The ({\bf irreducible})  {\bf characters} %\footnote{xxxx}
\index{characters of a finite abelian group}
\index{irreducible characters of a finite abelian group}
 of $G$ are the group homomorphisms
$\xi:G\to\CC\setminus\{0\}$, where $\CC\setminus\{0\}$ is a group under multiplication.
Here we think of characters %them 
as vectors $\xi\in\CC^G$ (with the Euclidean
inner product), which satisfy
\begin{equation}
\label{charhom}
\xi(gh)=\xi(g)\xi(h), \qquad\forall g,h\in G. 
\end{equation}
The set of irreducible characters of the abelian group $G$
is denoted by $\hat G$.

The characters $\hat G$ form a
group under the multiplication $(\xi\eta)(g):=\xi(g)\eta(g)$,
which is called the {\bf character group}. \index{character group}
The character group $\hat G$ is isomorphic to $G$.
For $\chi\in\hat G$, (\ref{charhom}) implies that $\chi(g)$ is a $|G|$--th root of unity,
and so the inverse of $\chi$ satisfies
%and so its inverse satisfies
%Thus inverse of $\chi\in\hat G$ is
\begin{equation}
\label{chargrpinv}
\chi^{-1}(g)={1\over\chi(g)}=\overline{\chi(g)}. 
\end{equation}
The square matrix with the irreducible characters of $G$ as rows
is referred to as the {\bf character table} of $G$.
For example, if $G=\langle a\rangle$
is the cyclic group of order $n$, with its elements ordered
$1,a,\ldots,a^{n-1}$, then its character table is given
by (\ref{cycordern}). 

The harmonic frames for $\Cd$ given by $G$ 
can all be described (up to unitary equivalence) in two equivalent 
ways:
\begin{enumerate}
\item
By a choice of $d$ characters 
$\{\xi_1,\ldots,\xi_d\}\subset\hat G$ ($d$ rows of the character table), i.e.,
\begin{equation}
\label{chardefharm}
\Psi_{\{\xi_1,\ldots,\xi_d\}} = (v_g)_{g\in G}, \qquad
v_g:=(\xi_j(g))_{j=1}^d\in\Cd.
\end{equation}
\item %or, equivalently (by Pontryagin duality) a choice of $d$ group
By a choice of $d$ group elements $J\subset G$ ($d$ columns of the character table), i.e.,
\begin{equation}
\label{groupdefharm}
\Phi_J = (\xi|_J)_{\xi\in\hat G}, \qquad \xi|_J\in\CC^J\cong\Cd.
\end{equation}
\end{enumerate}
Many properties of harmonic frames are easy to describe in
the second presentation. In particular:
\begin{itemize}
\item $\Phi_J$ has distinct vectors if and only if $J$ generates $G$.
\item $\Phi_J$ is a real frame
if and only  $J$ is closed under taking inverses.
\item If $\gs$ %:G\to G$ 
is an automorphism of $G$ and $K=\gs J$, then
$\Phi_J$ and $\Phi_K$ are unitarily equivalent. \\
In this case we say that $\Phi_J$ and $\Phi_K$ are 
{\bf unitarily equivalent via an automorphism}.
\end{itemize}

Note that if $J$ does not generate $G$, then $\Phi_J$ has $| \langle J \rangle |$ distinct vectors (where $\langle J \rangle$ is the group generated by $J$), each occuring $| G : \langle J \rangle |$ times.  Moreover, the frame obtained by taking one of each distinct vector is a harmonic frame associated to $\langle J \rangle$.

We now focus on the harmonic frames for $G=\ZZ_n$, 
which are said to be {\bf cyclic} (harmonic) frames. 
In this case, the automorphism group of $G$ (and hence $\hat G$) 
has a particularly simple form: each automorphism corresponds
to a unit $a\in\ZZ_n^*$ via
$$ \gs g= a g, \quad \forall g\in G, \qquad\gs\in\Aut(G), $$
$$ \tau \chi = \chi^a, \quad\forall \chi\in\hat G, \qquad
\tau\in\Aut(\hat G). $$
We say that %two 
$d$--element subsets $\{\xi_1,\ldots,\xi_d\}$
and $\{\eta_1,\ldots,\eta_d\}$ of $\hat\ZZ_n$, or
$d$--elements subsets $K$ and $J$ of $\ZZ_n$, are
{\bf multiplicatively equivalent} if there is an 
automorphism mapping one to the other, i.e.,
$$ \{\xi_1,\ldots,\xi_d\}=\{\eta_1,\ldots,\eta_d\}^a =
\{\eta_1^a,\ldots,\eta_d^a\}, \qquad
K=aJ. $$ 
This holds if and only if the cyclic harmonic frames they determine are unitarily equivalent via an automorphism.  Because of this, we shall sometimes apply the term `multiplicatively equivalent' to a pair of frames to mean they are unitarily equivalent via an automorphism.  An {\bf exceptional equivalence} is a unitary equivalence between 
cyclic harmonic frames given by %which are given by 
sets of characters
(or group elements) which are not multiplicatively equivalent.  
An {\bf exceptional frame} is one having an exceptional equivalence 
with another frame.

%If there is a unitary equivalence between cyclic harmonic 
%frames, and the frames are not unitarily equivalent
%via an automorphism (i.e., the subsets defining them
%are not multiplicatively equivalent), then the (unitary) equivalence
%and the frames are said to be {\bf exceptional}. 

\section{The number of cyclic harmonic frames}

Calculations of \cite{CW11} indicate
that {\it most cyclic harmonic frames are not exceptional},
i.e., unitarily equivalent cyclic harmonic frames 
usually come from multiplicatively equivalent subsets 
(this is proved in Proposition \ref{excount}).
This is the basic principle underlying our results.

Let $h_{n,d}$ be the number of unitarily inequivalent cyclic harmonic frames of
$n$ distinct vectors for $\Cd$. 
The main result is the following, which gives the growth of
$h_{n,d}$ (for $d$ fixed).

\begin{theorem}
\label{framecount}
If $d = 1$, then $h_{n,d} = 1$ for all $n$.  If  $d \ge 2$, then
for all $\epsilon >0$, we have
\begin{equation}
\label{hndexactasymp}
h_{n,d} = \frac{n^d}{d! \varphi(n)} 
\prod_{p | n} (1 - p^{-d}) 
\bigl(1 + O_\epsilon(n^{-1 + \epsilon})\bigr), \qquad n\to\infty, 
\end{equation}
where $\varphi$ is Euler's totient function, and the product
is over the prime factors $p$ of $n$.
\end{theorem}

Throughout, we use asymptotic notation, e.g., 
$a_n = O_\epsilon(b_n)$ means $|a_n|\le C b_n$ as $n\to\infty$,
where $b_n\ge0$ and $C$ is a constant depending only on $\epsilon$.
For $a_n\ge0$, $b_n\ge0$, we write
\begin{align*}
a_n\ll b_n & \Iff a_n=O(b_n), \cr
a_n\approx b_n  & \Iff a_n=O(b_n),\ b_n=O(a_n).
\end{align*}
The {\it Euler product formula} for the Riemann zeta function gives
\begin{equation}
\label{Eulerprodform}
0<{1\over\zeta(d)} 
%=\Bigl(\prod_{\hbox{$p$ prime}} {1\over1 - p^{-d}} \Bigr)^{-1}
= \prod_{p\ {\rm prime}} (1 - p^{-d}) 
< \prod_{p | n} (1 - p^{-d}) \le1, \qquad d=2,3,\ldots.
\end{equation}
Thus (\ref{hndexactasymp}) gives the asymptotic 
estimate
$$ h_{n,d} \approx {n^d\over\varphi(n)} \ge n^{d-1}, \qquad n\to\infty. $$

%$$ a_n\approx b_n \Iff \exists C_1,C_2>0: \quad
%C_1 a_n\le b_n \le C_2 a_n, \quad n\to\infty, $$
%$$ a_n\ll b_n \Iff a_n=O(b_n) \Iff \exists C>0: \quad
%a_n\le b_n, \quad n\to\infty. $$
%%{|a_n|\over |b_n|}\to 0, \quad n\to\infty$$

The proof of Theorem \ref{framecount}
(see the comments after Proposition  \ref{multiequivcountprop})
consists of two parts:
\begin{enumerate} 
\item We think of cyclic harmonic frames as being
given by subsets (or sequences) of $d$ characters of $\ZZ_n$, 
and hence by $n$--th roots of unity. 
This allows us to show that a pair of unitarily equivalent 
frames gives a torsion point on the torus $\TT^{2d}$ that 
satisfies certain algebraic equations. By using results about
the solutions of algebraic equation in roots of unity, 
we show that
the number exceptional cyclic harmonic frames grows slower than those which 
aren't (Proposition \ref{excount}).

\item In view of Proposition \ref{excount}, it suffices to count the
cyclic harmonic frames up to unitary equivalence via an
automorphism.
This we do by (Burnside) counting the $d$--element subsets of $\ZZ_n$
up to multiplicative equivalence
(Proposition \ref{multiequivcountprop}).
\end{enumerate}

We now give the arguments for each part 
(\S\ref{numexceptframesect} and \S \ref{multequivcntsect}).

\section{Torsion points and exceptional cyclic %harmonic 
frames}
\label{numexceptframesect}

%\section{Estimates of the number of harmonic frames}
%\label{numharmframes}

%Let $\go$ be the primitive $n$--th root of unity
%$$\go:=e^{2\pi i\over n}.$$
Since any character $\xi$ of $\ZZ_n$ satisfies
$$ \xi(k)=\xi(1)^k, \qquad\forall k,$$
%a character $\xi$ 
choosing $\xi$ is equivalent to choosing the $n$--th root of
unity $\xi(1)$. Thus a choice of characters
(giving a cyclic frame) corresponds to 
a choice of $n$--th roots via
$$ \{\xi_1,\ldots,\xi_d\} \Iff \{\xi_1(1),\ldots,\xi_d(1)\}. $$
In this section, we shall often think of a choice of $d$ characters 
(or $n$--th roots) as being an {\it ordered subset} 
$(\xi_1,\ldots,\xi_d)$. There are $d!$ such orderings of a given subset
$\{\xi_1,\ldots,\xi_d\}$.
%$$ \{\xi_1,\ldots,\xi_d\}\in\hat\ZZ_n \Iff 
%\{\xi_1(1),\ldots,\xi_d(1)\}\in\{1,\go,\dots,\go^{n-1}\}. $$
% $\{\xi_1,\ldots,\xi_d\}$

%, with the property $R_1$ generates the character group of erhich satisfy an additional condition. 
Two sets of $n$--th roots determining unitarily equivalent harmonic
frames satisfy the following. 

%The following simple condition captures most instances where two
%choices %of rows/characters 
%give equivalent frames.

%For fixed $d$, we will show that the number of harmonic frames of 
%$$n$ vectors for $\Cd$ grows like $n^{d-1}$.

\begin{lemma}
\label{frametopoint}
Let $G$ be a finite abelian group.
If $(v_g)_{g\in G}$ is a harmonic frame, then
$$ \inpro{v_{a+c},v_{b+c}}
= \inpro{v_{a},v_{b}}, \qquad\forall a,b,c\in G. $$
In particular, 
if $\{\xi_1,\ldots,\xi_d\},\{\eta_1,\ldots,\eta_d\}\subset\hat\ZZ_n$ 
give unitarily equivalent cyclic harmonic frames, then
for some $a\in\ZZ_n$, we have
\begin{equation}
\label{sumid}
\sum_{j=1}^d \xi_j(1) = \sum_{j=1}^d \eta_j(a).
\end{equation}
\end{lemma}

\begin{proof}
Let $(v_g)_{g\in G}$ and $(w_g)_{g\in G}$ be the harmonic
frames given by $\{\xi_j\}$ and $\{\eta_j\}$ via 
(\ref{chardefharm}).
Then we calculate
$$ \inpro{v_{a+c},v_{b+c}} 
= \sum_j \xi_j(a+c)\overline{\xi_j(b+c)}
= \sum_j \xi_j(a)\xi_j(c) \overline{\xi_j(b)} \overline{\xi_j(c)}
= \sum_j \xi_j(a) \overline{\xi_j(b)} 
=\inpro{v_{a},v_{b}}.  $$
If  $(v_g)_{g\in G}$ and $(w_g)_{g\in G}$  are unitarily
equivalent via (\ref{unitaryequivdef}), then
$$ \inpro{v_k,v_\ell}
=\inpro{Uw_{\gs k},Uw_{\gs \ell}}
=\inpro{w_{\gs k},w_{\gs \ell}}
=\inpro{w_{\gs k-\gs \ell},w_{0}}. $$
For $G=\ZZ_n$, taking $k=1$, $\ell=0$ %, $c=0$ 
above, gives
$\inpro{v_1,v_0}=\inpro{w_a,w_0}$, 
%$a:=\gs k-\gs \ell$,
$a:=\gs 1-\gs 0$,
which is (\ref{sumid}).
\end{proof}

Let $V\subset\CC^{2d}$ be the set of solutions to
\begin{equation}
\label{plane}
\sum_{j=1}^d z_j - \sum_{j=1+d}^{2d} z_j=0.
\end{equation}
By Lemma \ref{frametopoint},  %implies that 
every unitary equivalence between cyclic
harmonic frames of $n$ vectors for $\Cd$ gives a solution 
\begin{equation}
\label{zformula}
z=(\xi_1(1),\ldots,\xi_d(1),\eta_1(a),\ldots,\eta_d(a))\in V
\end{equation}
%to
%\begin{equation}
%\label{plane}
%\sum_{j=1}^d z_j - \sum_{j=1+d}^{2d} z_j=0
%\end{equation}
in $n$--th roots of unity.  
Moreover, one may easily produce such solutions by letting 
$z_{d+1}, \ldots, z_{2d}$ be a permutation of $z_1, \ldots, z_d$.  
We call a solution to (\ref{plane}) in roots of unity 
{\bf exceptional} if $z_{d+1}, \ldots, z_{2d}$ is not a 
permutation of $z_1, \ldots, z_d$.
%if it does not arise from a permutation.  

We shall prove 
(Lemma \ref{project}) 
that any exceptional equivalence between cyclic harmonic 
frames gives rise to an exceptional solution to (\ref{plane}).
%This allows us to prove:
This allows us to prove that the number of exceptional
cyclic harmonic frames is small
(Proposition \ref{excount}),
i.e., 
%of the $\approx n^d$ choices of $d$ characters of $\ZZ_n$
%only $\ll n^{d-1}$ give exceptional cyclic harmonic frames
$$\hbox{{\it Of the $\approx n^d$ choices 
%$\{\xi_1,\ldots,\xi_d\}\subset\hat\ZZ_n$
of $d$ characters of $\ZZ_n$,
$\ll n^{d-1}$ 
give exceptional cyclic harmonic frames. }} $$
This reduces the proof of Theorem \ref{framecount} to that of counting
the number of nonexceptional cyclic harmonic frames, which is done by 
counting $\ZZ_n^*$--orbits %via orbit counting
(see \S\ref{multequivcntsect}).

We will prove Proposition \ref{excount} by reducing it to a count 
of solutions to a linear equation in roots of unity.  
This is a well studied problem in the theory of Diophantine equations, 
and the set of solutions has a simple structure described by a theorem 
of Laurent \cite{L84}, 
which is a special case of the Mordell--Lang conjecture.
We now give the details.
%We now describe this approach in more detail.  

%Nonexceptional cyclic harmonic frames are unitarily equivalent 
%if and only if the subsets of $\ZZ_n$ giving them are
%multiplicatively equivalent. 
%Thus the number of nonexceptional cyclic harmonic frames
%is the number of orbits of the $d$--element subsets of $\ZZ_n$
%under the action of $\ZZ_n^*$.

%An exceptional equivalence is an equivalence between two cyclic harmonic frames whose sets of characters are not multiplicatively equivalent.  An exceptional frame is one having an exceptional equivalence with another frame.

%If exceptional equivalences between frames did not exist, 
%then $h_n$ would simply be the number of orbits of $\ZZ_n^*$ 
%on $d$-element subsets of $\widehat{\ZZ}_n$ whose elements 
%generate $\widehat{\ZZ}_n$, and this may be counted using Burnside's Theorem.  
%The following proposition shows that the number of choices of characters 
%that give exceptional frames is small.  
%This is the first main step in proving Theorem \ref{framecount}, 
%as it allows us to reduce the problem of estimating $h_n$ to one of 
%counting $\ZZ_n^*$-orbits.

Let $\TT^k$ be the {\bf $k$--torus}
$$ \TT^k := \{z\in\CC^k:|z_1|=\cdots=|z_k|=1\}, $$
which is a {\it compact abelian Lie group} under the group operation
$$ z\cdot w:=(z_1w_1,\ldots,z_kw_k).$$
A point on the torus of finite order is called a {\bf torsion point}.  
We denote the set of torsion points by $\TT^k_\textup{tors}$.  
We let $\TT^k[n]$ denote the set of $n$--torsion points
(torsion points of order $n$), 
which is the same as $k$--tuples of $n$--th roots of unity.  
There is a bijection between $d$--tuples of characters of $\ZZ_n$ and 
$\TT^d[n]$ sending $(\xi_1,\ldots,\xi_d)$ to $(\xi_1(1), \ldots, \xi_d(1) )$.

If the ordered subsets $(\xi_1,\ldots,\xi_d)$ and
$(\eta_1,\ldots,\eta_d)$ give equivalent harmonic frames,
then (\ref{zformula}) gives a point $z \in V \cap \TT^{2d}[n]$.
Note that $z$ depends on the choice of unitary equivalence between
the two frames made in the proof of Lemma \ref{frametopoint}.

%Let $V \subset \C^{2d}$ be the set of solutions to (\ref{plane}).  
%Let the ordered subsets $(\xi_1,\ldots,\xi_d)$ and 
%$(\eta_1,\ldots,\eta_d)$ give equivalent harmonic frames.  
%By Lemma \ref{frametopoint}, this gives a point $z \in V \cap \TT^{2d}[n]$.  
%Note that $z$ depends on the choice of unitary equivalence between 
%the two frames.  

%\subsection{Exceptional equivalences between frames}

%Throughout the proof of Proposition \ref{excount}, we shall consider subsets of $\widehat{\ZZ}_n$ to be ordered.  This is done for convenience, as it allows us to identify such a subset with a point on a torus as described below.  We shall denote ordered subsets of elements of $\widehat{\ZZ}_n$ by $\Xi$ and $\Theta$.

The basic result that allows us to reduce from counting frames to 
counting solutions to equations in roots of unity is the following.

\begin{lemma}
\label{project}
If the cyclic harmonic frame given by $(\xi_1(1),\ldots\xi_d(1))\in\TT^d[n]$
has distinct vectors and
is exceptional, then there is an exceptional point
$$z = (\xi_1(1), \ldots, \xi_d(1),z_{d+1},\ldots,z_{2d})\in V \cap\TT^{2d}[n]. $$
% Suppose that the frame associated to $\Xi \in \TT^d[n]$ is exceptional.  Then there is an exceptional point in $V \cap \TT^{2d}[n]$ that projects to $\Xi$.
\end{lemma}

\begin{proof} By hypothesis, there is an ordered set $(\eta_1, \ldots, \eta_d)$ of characters of $\ZZ_n$ such that $\{\xi_1,\ldots,\xi_d\}$ and $\{\eta_1,\ldots,\eta_d\}$ are multiplicatively inequivalent, but the frames they give are unitarily equivalent (and hence both have distinct vectors).  Then by Lemma \ref{frametopoint},
there is an $a\in\ZZ_n$ satisfying
(\ref{sumid}),
%$$ \sum_{j=1}^d \xi_j(1) = \sum_{j=1}^d \eta_j(a), $$
and so 
$$ z := (\xi_1(1),\ldots,\xi_d(1),\eta_1(a),\ldots,\eta_d(a)) 
\in V \cap \TT^{2d}[n]. $$
%We now show that $z$ is exceptional.
Suppose
(by way of contradiction) 
that $z$ is not exceptional, i.e., $z_{d+1},\ldots,z_{2d}$
is a permutation of $z_1,\ldots,z_d$. Then 
%$\{\xi_j(1)\}=\{\eta_j(a)\}=\{\eta_j(1)\}^a$.
after a reordering
$\xi_j(1)=\eta_j(a)=\eta_j(1)^a$, $\forall j$,
so that 
%and hence 
$\{\xi_j\}=\{\eta_j\}^a$. 
The frame given by $\{\xi_j\}$ has distinct vectors,
and so $\{\xi_j\}$ generates $\hat\ZZ_n$.
Thus $\{\xi_j(1)\}=\{\eta_j(1)\}^a$
generates the $n$--th roots of unity, and $a$ must be a unit.
This implies that
$\{\xi_j\}$ and $\{\eta_j\}$
are multiplicatively equivalent,
a contradiction. Thus $z$ is an exceptional point.
\end{proof}

The set of solutions to (\ref{plane}) in torsion points is described by the 
Mordell--Lang conjecture for tori, proved by Laurent \cite[Theorem 2]{L84} (see also the Conjecture on page 299). This states the following:

\begin{theorem}
\label{Laurent}
Let $f$ be a (holomorphic) polynomial on $\CC^k$ with zero set
$$Z(f) := \{z\in\CC^k:f(z)=0\}. $$
Then there are a finite number of (topologically) closed, connected subgroups
$T_1,\ldots,T_m$ of $\TT^k$, and points $p_1,\ldots,p_m$ in $\TT^k_\textup{tors}$, such that
\begin{equation}
\label{Laurentdecomposition}
\TT_\textup{tors}^k \cap Z(f) = \TT_\textup{tors}^k \cap \bigcup_{j=1}^m (p_j \cdot T_j), \qquad
p_j\cdot T_j\subset Z(f), \ \forall j. 
\end{equation}
\end{theorem}

We note that this differs slightly from %the statement of 
Laurent's theorem,
which %Laurent 
states that there is a finite collection of Zariski-closed subgroups 
$H_j \subset (\C^*)^k$ and $p_j \in \TT^k_\text{tors}$, such that
$$
\TT_\textup{tors}^k \cap Z(f) =  \bigcup_{j=1}^m 
\bigl(\TT_\textup{tors}^k \cap (p_j \cdot H_j) \bigr), \qquad
p_j\cdot H_j\subset Z(f), \ \forall j. 
$$
However, we have $\TT_\textup{tors}^k \cap (p_j \cdot H_j) \subset \TT^k$, 
and $H_j \cap \TT^k$ is a (topologically) closed subgroup of $\TT^k$, 
which must be a finite union of translates of a connected closed subgroup 
by torsion points. 
From this, it is easy to deduce Theorem \ref{Laurent}.

We now apply Theorem \ref{Laurent} to %control 
$V \cap \TT^{2d}_\textup{tors}$.  
Let $\pi$ be the projection of $\TT^{2d}$ onto the first $d$ components, i.e.,
 $$\pi:\TT^{2d}=\TT^d\times\TT^d\to\TT^d:(z,w)\mapsto z. $$
For any $\sigma \in S_d$, 
let $T_\sigma = \{ (z_1, \ldots, z_d, z_{\gs 1}, \ldots, z_{\gs d}) : 
|z_j| = 1 \} \subset \TT^{2d}$, so that

$$
\TT^{2d}_\textup{tors} \cap \bigcup_{\sigma \in S_d} T_\sigma
$$
is the set of solutions to (\ref{plane}) in roots of unity that 
are not exceptional.
%arise from a permutation.

\begin{lemma}
\label{pointcontrol}
There are a finite number of (topologically) closed, connected subgroups
$T_1,\ldots,T_m$ of $\TT^{2d}$ satisfying $\dim \pi(T_j) < d$
for all $j$, 
and points $p_1,\ldots,p_m$ in $\TT^{2d}_\textup{tors}$, such that
\begin{equation}
\TT_\textup{tors}^{2d} \cap V 
= \bigcup_{\sigma \in S_d} (\TT_\textup{tors}^{2d} \cap T_\sigma) \cup 
\bigcup_{j=1}^m (\TT_\textup{tors}^{2d} \cap p_j \cdot T_j).
\end{equation}
\end{lemma}

\begin{proof}
We apply Theorem \ref{Laurent} with $Z(f) = V$.
We now determine the possible tori $T_j$ in (\ref{Laurentdecomposition}) 
that can satisfy $p_j \cdot T_j \subset V$.  
%If $\alpha \in \ZZ^k$ and $z = (z_1, \ldots, z_k) \in \TT^k$, 
%we let $z^\alpha$ denote $z_1^{\alpha_1} \ldots z_k^{\alpha_k}$.  
For $\alpha \in \ZZ^k$ and $z \in \TT^k$, we define
$z^\alpha:=z_1^{\alpha_1} \ldots z_k^{\alpha_k}$.  

A connected subgroup $T\subset\TT^{2d}$ of dimension $k$ is the image of a map
$$ \TT^k\to\TT^{2d} : z=(z_1,\ldots,z_k)\mapsto(z^{\ga_1},\ldots,z^{\ga_{2d}}), \qquad
\ga_j\in\ZZ^k, $$
where the $\ga_j$ must span $\RR^k$, since the map has a discrete kernel.
If $p=(\go_1,\ldots,\go_{2d}) \in \TT^{2d}$, 
then %it follows that 
$p\cdot T$ is the image of the map
%$$ \psi: z \mapsto(\go_1 z^{\ga_1},\ldots,\go_{2d} z^{\ga_{2d}}). $$
$z \mapsto(\go_1 z^{\ga_1},\ldots,\go_{2d} z^{\ga_{2d}}).$
Thus for the linear polynomial (\ref{plane}) to vanish on $p\cdot T$, 
we must have
$$ \sum_{j=1}^d \go_j z^{\ga_j} -
\sum_{j=d+1}^{2d} \go_j z^{\ga_j} = 0, \qquad
\forall z\in\TT^k. $$
For any exponent $\beta$, the sum of the coefficients of $z^\gb$ above must be 0.  It follows that any $\gb$ occuring as an $\ga_j$ 
must occur at least twice, 
and so there can be at most $d$ distinct exponents.
%It follows that the left hand side of this equation must vanish as a polynomial in $z_1, \ldots, z_k$ and their inverses.  Hence any exponent occurring as an $\ga_i$ must occur at least twice,
%so there can be at most $d$ distinct exponents.  
Thus the $\alpha_j$ can span $\RR^k$ only if $k\le d$.  
If $k < d$, then clearly $\dim \pi( T) < d$.

If $k=d$, then every exponent occurs exactly twice.  
If some exponent occurs twice in the first sum, i.e.,
$\alpha_j = \alpha_k$ for some $1 \le j \neq k \le d$,
then the function $z_j / z_k$ is constant on $T$,
which implies that $\dim \pi( T) < d$.
%.  This implies that the projection
%of $T$ onto $\TT^d$ has dimension $< d$.
If no exponent occurs twice in the first sum, 
then each exponent occurs once in each sum,
i.e., $\ga_{d+j}=\ga_{\gs j}$, $1\le j\le d$, for some $\sigma \in S_d$.
%This means there is some $\sigma \in S_d$ such that 
%$\alpha_{d + i} = \alpha_{\sigma(i)}$ for all $1 \le i \le d$.  
This implies that $T \subset T_\sigma$, and as both are 
$d$--dimensional and connected, we have that $T = T_\sigma$.  
We must also have $\omega_{d+j} = \omega_{\sigma j}$, $1 \le j\le d$, 
so that $p \in T$ and $p \cdot T = T_\gs$.

We have shown that

\begin{equation*}
\TT_\textup{tors}^{2d} \cap V 
= \bigcup_{\sigma \in X} (\TT_\textup{tors}^{2d} \cap T_\sigma) \cup 
\bigcup_{j=1}^m (\TT_\textup{tors}^{2d} \cap p_j \cdot T_j),
\end{equation*}
where $X \subset S_d$ and $\dim \pi(T_j) < d$ for all $j$.  The inclusion

\begin{equation*}
\bigcup_{\sigma \in S_d} (\TT_\textup{tors}^{2d} \cap T_\sigma)  \subset \TT_\textup{tors}^{2d} \cap V
\end{equation*}
means that we can take $X = S_d$, which completes the proof.

\end{proof}

\begin{proposition}
\label{excount}
The number of choices of $d$ characters of $\ZZ_n$ which lead to 
exceptional cyclic harmonic frames with distinct vectors is $\ll n^{d-1}$ as $n\to\infty$.
\end{proposition}

\begin{proof}
Let $\mathcal{E} \subset \TT^d[n]$ be the set of $d$-tuples of
characters of $\ZZ_n$ (viewed as $n$--th roots) that give  exceptional 
cyclic harmonic frames with distinct vectors,  
and let $\widetilde{\mathcal{E}} \subset V \cap \TT^{2d}[n]$ 
be the set of exceptional points. By Lemma \ref{project} 
and Lemma \ref{pointcontrol}, we have
$$ \mathcal{E} \subset \pi( \widetilde{\mathcal{E}} )
%\pi( \widetilde{\mathcal{E}} ) 
\subset \TT^d[n] \cap \bigcup_{j=1}^m \pi(p_j \cdot T_j), $$
with $\dim \pi(T_j) < d$ for all $j$, which implies that $| \TT^d[n] \cap \pi(p_j \cdot T_j) | \ll n^{d-1}$.  Since the collection of translates $p_j \cdot T_j$ only depends on $d$, 
we obtain $| \mathcal{E} | \ll n^{d-1}$ as required.
\end{proof}

%We now prove Proposition \ref{excount}.  Let $\mathcal{E} \subset \TT^d[n]$ be the set of exceptional frames, and let $\widetilde{\mathcal{E}} \subset V \cap \TT^{2d}[n]$ be the set of exceptional points.  Lemma \ref{project} gives $\mathcal{E} \subset \pi( \widetilde{\mathcal{E}} )$.  Lemma \ref{pointcontrol} gives
%$$
%\pi( \widetilde{\mathcal{E}} ) \subset \TT^d[n] \cap \bigcup_{j=1}^m \pi(p_j \cdot T_j)
%$$
%with $\dim \pi(T_j) < d$ for all $j$, which implies that $| \TT^d[n] \cap \pi(p_j \cdot T_j) | \ll n^{d-1}$.  
%Combining these gives $| \mathcal{E} | \ll n^{d-1}$ as required.

%\section{Counting the cyclic harmonic frames up to multiplicative equivalence}
\section{Counting the nonexceptional cyclic harmonic frames}
\label{multequivcntsect}

Let $m_{n,d}$ be the number of %of nonexceptional 
cyclic harmonic frames of $n$ distinct vectors for $\Cd$, 
up to unitary equivalence via an automorphism,
i.e.,
the number of $d$--element subsets which generate $\ZZ_n$,
% and give nonexceptional frames, 
up to multiplicative equivalence.  (In this section, it will be convenient to work with group elements rather than characters.)
We will prove Theorem \ref{framecount}
by calculating $m_{n,d}$,
%and 
then using Proposition \ref{excount} to conclude 
$$ h_{n,d} \approx m_{n,d}, \qquad n\to\infty. $$
Since all elements which generate $\ZZ_n$ are multiplicatively equivalent,
we have $h_{n,1}=m_{n,1}=1$.  We may therefore assume that $d \ge 2$ for the rest of this section. 

Let $\cY_\textup{gen}$ be the set of all $d$--element subsets
which generate $\ZZ_n$ (i.e. give cyclic
harmonic frames with distinct vectors), and $\cY_\textup{ex} \subset \cY_\textup{gen}$
be the subset which gives exceptional frames. Then $m_{n,d}$ is the 
number of $\ZZ_n^*$--orbits of 
$\cY_\textup{gen}$ under
the multiplicative action of $\ZZ_n^*$.

If $S$ is a collection of $d$--element subsets of $\ZZ_n$, and $S$ is 
stable under the action of $\ZZ_n^*$, we denote the set of its
orbits by $S/\ZZ_n^*$, and the set of its elements
fixed by $a\in\ZZ_n^*$ by $\Fix(a)=\Fix(a,S)$.

\begin{lemma}
\label{orbits}
We have the bounds
$$ |(\cY_\textup{gen}\setminus \cY_\textup{ex})/\ZZ_n^*|
\le h_{n,d} \le 
m_{n,d} = | \cY_\textup{gen} / \ZZ_n^* |. $$
%$$ | \cY_\textup{gen} / \ZZ_n^* | \ge h_n \ge | (\cY_\textup{gen} \setminus \cY_\textup{ex}) / \ZZ_n^* |.  $$
\end{lemma}

\begin{proof} We observe $h_{n,d}$ is the number of equivalence classes
in $\cY_\textup{gen}$ under the equivalence relation given by unitary
equivalence of the corresponding frames. Since each equivalence
class is stable under the  action of $\ZZ_n^*$, the number of such
classes is at most the number of $\ZZ_n^*$--orbits, which gives the
upper bound.
%Any harmonic frame arises from an element of $\cY_\textup{gen}$.  Moreover, $h_n$ is equal to the number of classes in $\cY_\textup{gen}$ under the  equivalence relation given by unitary equivalence of frames.  Each equivalence class is stable under the action of $\ZZ_n^*$, and so the number of such classes is at most the number of $\ZZ_n^*$ orbits.  This gives the upper bound.  
By definition of $\cY_\textup{ex}$, 
the unitary equivalence classes in $\cY_\textup{gen} \setminus\cY_\textup{ex}$ 
are exactly the $\ZZ_n^*$--orbits, which gives the lower bound.
\end{proof}

We now estimate the sizes of 
$\cY_\textup{gen}$ and $\cY_\textup{gen} \setminus \cY_\textup{ex}$.

\begin{lemma}
\label{gencount}
We have
%$$ | \cX_\textup{gen} | = n^d \prod_{p | n} (1 - p^{-d}).  $$
%\begin{equation}
%\label{gendist}
%| \cY_\textup{gen} | = \frac{n^d}{d!} \prod_{p | n} (1 - p^{-d}) + O(n^{d-1}).
%\end{equation}
%\begin{equation}
%\label{nonexdist}
%| \cY_\textup{gen} \setminus \cY_\textup{ex} | = \frac{n^d}{d!} \prod_{p | n} (1 - p^{-d}) + O(n^{d-1}).
%\end{equation}
\begin{equation}
\label{numgendist}
|\cY_\textup{gen}|,| \cY_\textup{gen} \setminus \cY_\textup{ex} | 
= \frac{n^d}{d!} \prod_{p | n} (1 - p^{-d}) + O(n^{d-1}), \quad
n\to\infty,
\end{equation}
where the product is over the prime factors $p$ of $n$.
\end{lemma}

\begin{proof}
It is convenient to work with the ordered subsets of $\ZZ_n$.
Let $\cX=\ZZ_n^d$ be the set of $d$--tuples of elements of $\ZZ_n$,
$\cX_\textup{gen}$ be the subset of those
whose elements generate $\ZZ_n$, and 
$\cX_\textup{dist}$ be the subset of those whose elements are all distinct.
Clearly,
$$ |\cY_\textup{gen}| = |\cX_\textup{dist}\cap\cX_\textup{gen}|/d!. $$
The size of $\cX_\textup{gen}$ is {\it Hall's $d$--th Eulerian function}
($d=1$ gives Euler's totient function $\varphi(n)$).
We now calculate $|\cX_\textup{gen}|$ by using inclusion--exclusion counting.
Let $\cX(m) \subset \cX$ be the collection of $d$--tuples 
all of whose elements lie in $m\ZZ_n$. If some $d$--tuple %element of $\cX$ 
does not generate $\ZZ_n$, then its elements must be contained
in some maximal proper subgroup $p\ZZ_n$, $p | n$, and so we have
$$ \cX_\textup{gen} = \cX \setminus \bigcup_{p | n} \cX(p).  $$
It is easy to see, 
that if $p_1, \ldots, p_k$ are distinct primes dividing $n$, then
$$ \Bigl|\bigcap_{j=1}^k\cX(p_j)\Bigr| = |\cX(p_1p_2 \cdots p_k)| 
= \bigl(n / (p_1p_2 \cdots p_k) \bigr)^d. $$
%Indeed, if $\cX$ does not generate $\widehat{\ZZ}_n$ then its elements must be contained in some maximal proper subgroup, and these subgroups are exactly $p \widehat{\ZZ}_n$ for $p | n$.  Inclusion-exclusion counting then gives
Thus, inclusion--exclusion counting gives
\begin{align*}
%\label{inex}
| \cX_\textup{gen} | &= | \cX | - \sum_{p | n} |\cX(p)| 
+ \sum_{ \substack{ p_1, p_2 | n \\ p_1 \neq p_2 } } | \cX(p_1) \cap \cX(p_2)|
-  \cdots \cr
& = n^d \prod_{p | n} (1 - p^{-d}). 
\end{align*}
Since $|\cX \setminus \cX_\textup{dist}|=n^d-n(n-1)\cdots(n-d+1) \ll n^{d-1}$, 
we have
$$ d! |\cY_\textup{gen} |= | \cX_\textup{gen} \cap \cX_\textup{dist} | 
= n^d \prod_{p | n} (1 - p^{-d}) + O(n^{d-1})
$$
which gives the estimate for $|\cY_\textup{gen} |$. % (\ref{numgendist}).
Because $|\cY_\textup{ex}| \ll n^{d-1}$ by Proposition \ref{excount}, 
we also have the estimate for $|\cY_\textup{gen}\setminus\cY_\textup{ex} |$. 
%(\ref{nonexdist}).
\end{proof}

We now count the number of orbits for the action of $\ZZ_n^*$
on $\cY_\textup{gen}$ and $\cY_\textup{gen} \setminus \cY_\textup{ex}$. 
Recall Burnside's Theorem, which states that if $G$ is a finite group
acting on a finite set $S$ then the number of orbits is
\begin{equation}
\label{Burnsidect}
|S/G| = {1\over|G|} \sum_{a\in G} |\Fix(a,S)|. 
\end{equation}

%

%\begin{theorem}
%\label{Burnside}
%Let a finite group $G$ act on a finite set $S$.  
%The number of orbits of $G$ in $S$ is
% $$ \frac{1}{|G|} \sum_{\varphi \in G} \textup{fix}(\varphi).  $$
%\end{theorem}

We shall combine Burnside's Theorem with the following bound 
for $|\Fix(a, \cY_\textup{gen})|$.

\begin{lemma}
\label{fixbound}
Let $a \in \ZZ_n^*$. Then 
\begin{enumerate}
\item
$|\Fix(a, \cY_\textup{gen})| \le n^{d-1}$ for $a \ne 1$.
\item
$|\Fix(a, \cY_\textup{gen})| \le n^{d-2}$ for $a^2 \ne 1$.
\end{enumerate}
\end{lemma}

\begin{proof} 
Let $A\in\Fix(a, \cY_\textup{gen})$. 
We note that the elements of $A\in\cY_\textup{gen}$ generate $\ZZ_n$.

If $a\ne 1$, % and $b\in A$, 
then $H=\{b\in\ZZ_n:ab=b\}$ is a proper subgroup of $\ZZ_n$,
and so there is some $b\in A$, $b\not\in H$.
As the choice of $b$ determines a second element $ab \in A$, 
the number of choices for $A$ is less than
$n\cdot n^{d-2}=n^{d-1}$.
Similarly, 
if $a^2\ne 1$ (so $a\ne1$), then %we claim there are three distinct
$H=\{b\in\ZZ_n:a^2b=b\}$ is a proper subgroup of $\ZZ_n$,
so that $\{b,ab,a^2b\}$ are distinct elements of $A$ 
for $b\in A$, $b\notin H$.  It follows that if $d = 2$ then $\Fix(a, \cY_\textup{gen}) = \emptyset$, and if $d \ge 3$ the number of choices for $A$ is less than $n\cdot n^{d-3}=n^{d-2}$.
%Suppose $a^2 \neq 1$.  
%It follows that the set of $\xi \in \widehat{\ZZ}_n$ such that $(a^2-1) \xi = 0$ is a proper subgroup of $\widehat{\ZZ}_n$, which we call $N$.  We may obtain every $A$ counted by $|\Fix(a,\cY_\text{gen})|$ by the following procedure.  Choose $\xi \in \widehat{\ZZ}_n$, and add it to $A$ together with its orbit under $\ZZ_n^*$.  Repeat this in such a way that one obtains a set of exactly $d$ elements.  Because $A$ generates $\widehat{\ZZ}_n$, at one point we must choose a character $\xi \notin N$, and this determines at least two other elements of $A$ because $\xi$, $a\xi$, and $a^2 \xi$ are all distinct.  We therefore have at most $n^{d-2}$ choices in constructing $A$, which gives the claim.
%
%If $a \neq 1$, we proceed in the same way with $N$ equal to the subgroup of $\xi$ such that $(a-1)\xi = 0$.  In this case, the character $\xi \notin N$ only determines one other element of $A$, namely $a \xi$, giving a saving of 1 in the exponent.
\end{proof}

\begin{proposition}
\label{multiequivcountprop}
For each $\gep>0$, we have
\begin{equation}
\label{mainasympest}
|\cY_\textup{gen}/\ZZ_n^*|,|(\cY_\textup{gen}\setminus\cY_\textup{ex})/\ZZ_n^*|
= {1\over d!\varphi(n)} 
\bigl( n^d \prod_{p | n} (1 - p^{-d}) + O_\epsilon(n^{d-1 + \epsilon})\bigr),
\quad n\to\infty.
%= \frac{n^d}{d!} \prod_{p | n} (1 - p^{-d}) + O(n^{d-1}).
\end{equation}
\end{proposition}

\begin{proof} 
Since $|\ZZ_n^*|=\varphi(n)$, counting the 
$\ZZ_n^*$--orbits of $\cY_\textup{gen}$ by
Burnside's Theorem (\ref{Burnsidect}) gives
$$
| \cY_\textup{gen} / \ZZ_n^* | = \frac{1}{\varphi(n)} \sum_{a \in \ZZ_n^*} 
|\Fix(a, \cY_\textup{gen})|.
$$
Partition $\ZZ_n^*$ into 
$\{ 1 \}$, $B=\{ a : a^2 = 1, a \neq 1 \}$ and $\{ a : a^2 \neq 1\}$, and apply Lemma \ref{fixbound} to the last two sets to obtain
$$ | \cY_\textup{gen} / \ZZ_n^* | \le \frac{1}{\varphi(n)} 
\bigl( | \cY_\textup{gen} | + |B| n^{d-1} + \varphi(n) n^{d-2} \bigr).  $$ 
We recall that $\ZZ_{p^m}^*$ is cyclic for $p$ an odd prime,
and $\ZZ_{2^m}^*$ is a product of at most two cyclic groups.
By the Chinese Remainder Theorem and 
the structure of $\ZZ_{p^m}^*$ for $p$ prime, we have
$$ |B| < | \{ a \in \ZZ_n^* : a^2 = 1 \} | \le 2^{\omega(n) + 1} $$
where $\omega(n)$ is the number of prime factors of $n$.
%, and the $+1$ comes from the fact that $\ZZ_{2^m}^*$ is not cyclic for $m \ge 2$.  
We see that $2^{\omega(n)}$ is at most the number of divisors $d(n)$ of $n$, and it is known that $d(n) \ll_\epsilon n^\epsilon$, 
see for instance \cite[Thm 13.12]{A76}.  Applying this and $\varphi(n) \le n$ gives
$$ | \cY_\textup{gen} / \ZZ_n^*| \le \frac{1}{\varphi(n)} 
\bigl( | \cY_\textup{gen} | + O_\epsilon( n^{d-1+\epsilon}) \bigr).
$$
Combining this with the estimate (\ref{numgendist}) 
for $| \cY_\textup{gen} |$ gives the upper bound
\begin{equation}
\label{orbupperbnd}
| \cY_\textup{gen} / \ZZ_n^* | 
\le \frac{1}{d! \varphi(n)} 
\bigl( n^d \prod_{p | n} (1-p^{-d})+O_\epsilon(n^{d-1+\epsilon})\bigr). 
\end{equation}
%which gives the upper bound in Theorem \ref{framecount}.

We estimate $|(\cY_\textup{gen} \setminus \cY_\textup{ex})/\ZZ_n^*|$ 
in the same way. Let $S=\cY_\textup{gen}\setminus \cY_\textup{ex}$ 
in %Burnside's theorem
(\ref{Burnsidect}) and take only the $a=1$ term to obtain
%only the ter  and dropping all terms with $a \neq 1$ gives
$$
| (\cY_\textup{gen} \setminus \cY_\textup{ex}) / \ZZ_n^* | \ge \frac{1}{\varphi(n)} | \cY_\textup{gen} \setminus \cY_\textup{ex} |.
$$
Combining this with the estimate (\ref{numgendist}) for 
$| \cY_\textup{gen}\setminus\cY_\textup{ex} |$
gives the lower bound 
\begin{equation}
\label{orblowerbnd}
| (\cY_\textup{gen} \setminus \cY_\textup{ex}) / \ZZ_n^* | \ge 
\frac{1}{d! \varphi(n)} 
\bigl( n^d \prod_{p | n} (1-p^{-d})+O_\gep(n^{d-1+\gep})\bigr). 
\end{equation}
Since $|(\cY_\textup{gen}\setminus\cY_\textup{ex})/\ZZ_n^* | 
\le | \cY_\textup{gen}  / \ZZ_n^* | $, 
the bounds (\ref{orbupperbnd}) and (\ref{orblowerbnd})
give (\ref{mainasympest}).
%in Theorem \ref{framecount}, and completes the proof.
\end{proof}

By Lemma \ref{orbits},
Proposition \ref{multiequivcountprop}
and (\ref{Eulerprodform}), we have
\begin{align*}
h_{n,d}, m_{n,d} 
& =  \frac{1}{d! \varphi(n)} 
\bigl( n^d \prod_{p | n} (1-p^{-d})+O_\gep(n^{d-1+\gep})\bigr) \cr
& =  \frac{n^d}{d! \varphi(n)} 
\prod_{p | n} (1-p^{-d}) \bigl( 1+O_\gep(n^{-1+\gep})\bigr) , 
\quad n\to\infty,
\end{align*}
which completes the proof of Theorem \ref{framecount}
(we already observed that $h_{n,1}=m_{n,1}=1$).

\section{Some examples}

The number of cyclic harmonic frames up to multiplicative 
equivalence can be calculated exactly by Burnside counting
(this can be done by a computer algebra package):
\begin{equation}
\label{mndform}
m_{n,d}=|\cY_\textup{gen}/\ZZ_n^*|
= {1\over\varphi(n)}\Bigl(|\cY_\textup{gen}|
+\sum_{a\in\ZZ_n^*\atop a \neq 1}|\Fix(a,\cY_\textup{gen})|\Bigr).
\end{equation}
%$$ m_{n,d}=|\cY_\textup{gen}/\ZZ_n^*|
%= {1\over\varphi(n)}\Bigl(|\cY_\textup{gen}|
%+\sum_{j=2}^d\sum_{a\in\ZZ_n\atop|a|=j}|\Fix(a,\cY_\textup{gen})|\Bigr). $$
This slightly over counts $h_{n,d}$ when there are exceptional frames.
Theorem \ref{framecount} gives the approximation
\begin{equation}
\label{hndapprox}
h_{n,d} \approx a_{n,d}:=\frac{n^d}{d! \varphi(n)} 
\prod_{p | n} (1 - p^{-d}), \qquad n\to\infty.
\end{equation}
This appears to give a good fit to $h_{n,d}$ and $m_{n,d}$
(even for small values of $n$), see Figure \ref{nharmgraph}.
\begin{figure}[h]
\centering
\includegraphics[width=14.0cm]{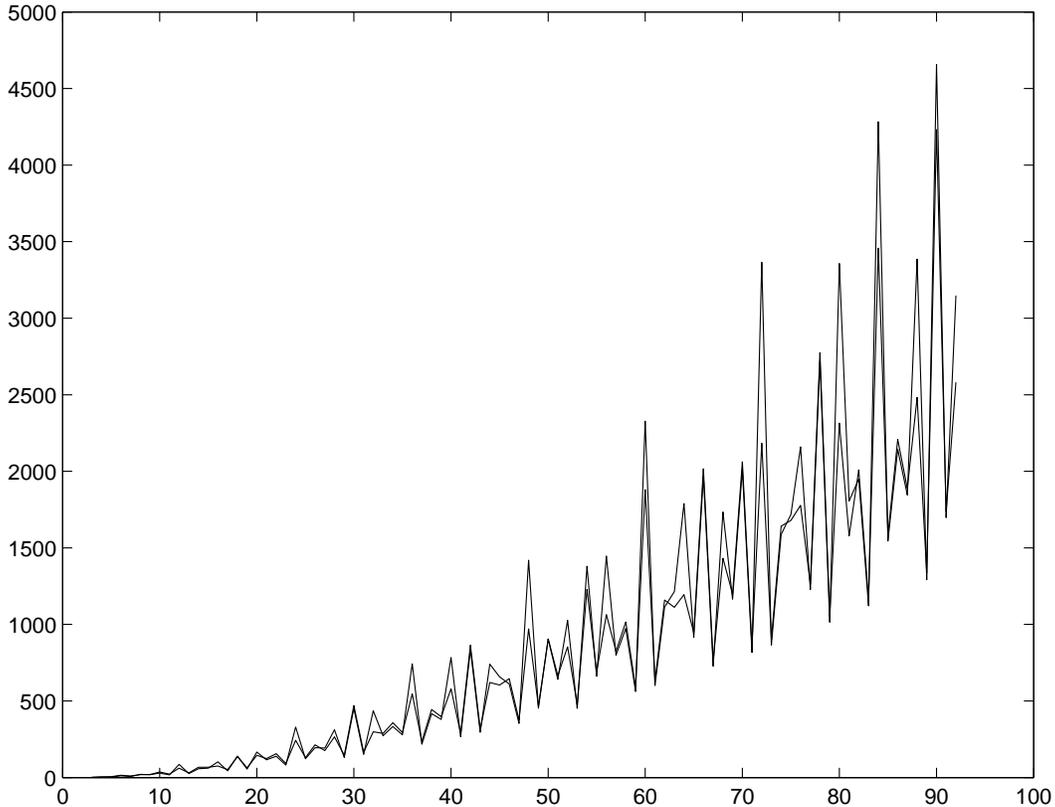}
\caption{The number of harmonic frames of $n$ distinct vectors for $\CC^3$
(including noncyclic frames) as calculated by \cite{HW06} together
with $a_{n,3}$.  }
\label{nharmgraph} 
\end{figure}

Since $\varphi(n)$ is multiplicative, with 
$$ \varphi(p^m)= p^m\Bigl(1-{1\over p}\Bigr), $$
it follows that
$$ a_{n,1}=1, \qquad
a_{n,2} = {n\over 2} \prod_{p | n} \Bigl(1 + {1\over p}\Bigr)\in\ZZ, 
\quad n>2. $$ 
For $d>2$, $a_{n,d}$ may not be an integer.

\begin{example}
For $d=2$,  %(and $d=1$), 
unitary equivalence and multiplicative equivalence are the same 
\cite{CW11}, i.e, $h_{n,2}=m_{n,2}$. Despite
$a_{n,2}$ being an integer, it does not always equal $m_{n,2}$. 
%This is because $a_{n,d}$ is obtained from (\ref{mndform}) by taking only the first term which is then over counted for $d>1$.
They first differ for $n=8$, when the units group has
three elements of order $2$, i.e., $3,5,7\in\ZZ_8^*$ with
%$$\{ 3, 5 \}, \{ 6, 7 \}, \{ 2, 3 \}, \{ 3, 6 \}, \{ 4, 5 \}, \{ 5, 6 \}, 
%\{ 0, 3 \}, \{ 3, 4 \}, \{ 1, 3 \}, \{ 1, 4 \}, \{ 1, 2 \}, \{ 2, 5 \},
%\{ 5, 7 \}, \{ 0, 1 \}, \{ 4, 7 \}, \{ 0, 5 \}, \{ 1, 5 \}, \{ 2, 7 \},
% \{ 1, 6 \}, \{ 3, 7 \}, \{ 0, 7 \}, \{ 1, 7 \} $$
$$ \Fix(3)=\bigl\{\{ 1, 3 \},\{ 5, 7 \}\bigr\}, \quad
\Fix(5)=\bigl\{\{ 1, 5 \}, \{ 3, 7 \}\bigr\} \qquad
\Fix(7)= \bigl\{\{ 3, 5 \}, \{ 1, 7 \}\bigr\}. $$
Since $\cY_\textup{gen}$ has $22$ elements, 
(\ref{mndform}) gives 
$a_{8,2}= 6 < 7= {1\over 4}(22+2+2+2)=m_{8,2}.$
%$$ a_{8,2}={1\over 2}|\cX\textup{gen}|/\varphi(n)
%={1\over 2}(22\cdot 2+4)/4 < 7= {1\over 4}(22+2+2+2)=m_{8,2}. $$
\end{example}

%$$ 2|\cY_\textup{gen}|
%=|\cX_\textup{gen}\cap\cX_\textup{dist}|
%=|\cX_\textup{gen}|-\varphi(n) $$

\begin{example} For $d>2$ there exist exceptional cyclic frames,
e.g., for $n=8$ vectors in $\CC^3$, there are $17$ multiplicative
equivalence classes of $3$--element subsets that generate $\ZZ_8$.
The classes
$$ \bigl\{\{1,2,5\},\{3,6,7\}\bigr\}, \qquad
\bigl\{\{1,5,6\},\{2,3,7\}\bigr\} $$
give exceptional frames (a frame from each class gives an 
exceptional pair), and it is easy to show that
$ h_{8,3}=16<17=m_{8,3}$.
\end{example}

\section{Harmonic frames with a prime number of vectors}

When $n$ is prime, then the $n-1$ primitive $n$--th roots 
of unity are a $\Q$--basis for the cyclotomic field $\Q[\go]$
that they generate. As a consequence, many of the inner products
between vectors of the (cyclic) harmonic frames are distinct,
and so there are no exceptional frames:
%and multiplicative and unitary equivalence are the same:

\begin{proposition} If $n = p$ is prime, then all harmonic frames 
of $n$ distinct vectors for $\Cd$
are cyclic, and multiplicative and unitary equivalence are the same for these frames, 
i.e., $h_{p,d}=m_{p,d}$. 
\end{proposition}

\begin{proof}  Let $\{\xi_1,\ldots,\xi_d\}$ and $\{\eta_1,\ldots,\eta_d\}$
be characters of $\ZZ_p$ giving unitarily equivalent harmonic
frames with distinct vectors. By Lemma \ref{frametopoint} there is some $a \in \ZZ_p$ with
\begin{equation}
\label{inprovals}
\sum_{j=1}^d \xi_j(1) = \sum_{j=1}^d \eta_j(a).
\end{equation}
We will show that after a reordering
$\xi_j(1)=\eta_j(a)$, $\forall j$.
The argument of Lemma \ref{project} then implies that the frames 
are multiplicatively equivalent.

Let $\go$ be a primitive
$p$--th root of unity,
and $a_j$ and $b_j$ be the number of 
times that $\omega^j$ occurs as a summand on the left hand and right hand
sides of (\ref{inprovals}). Then
$$ a_0+\cdots a_{p-1}= b_0+\cdots+b_{p-1}=d. $$
Since $1+\go+\cdots+\go^{p-1}=0$, we may replace each appearance
of $1$ in (\ref{inprovals}) by $-\go-\cdots-\go^{p-1}$, and
equate coefficients of the resulting $\Q$--linear combinations of 
primitive roots to obtain
$$ a_j-a_0=b_j-b_0, \qquad j=1,2,\dots,p-1. $$
Solving the above system of $p$ linear equations in $a_0,\ldots,a_{p-1}$
gives $a_j=b_j$, $\forall j$.
%after a reordering $\xi_j(1)=\eta_j(a)$, $\forall j$
\end{proof}

The above argument shows that harmonic frames 
that are unitarily inequivalent have
no inner product between distinct vectors in common.
%, then the that the multiset
%$\{ \go^{aj_1}+\go^{aj_2}+\cdots+\go^{aj_d}: a\in\ZZ_p, a\ne 0 \} $
%of inner products
%between distinct vectors in the frame given by 
%$J=\{j_1,\ldots,j_d\}\subset\ZZ_p$
%$$ \{ \go^{aj_1}+\go^{aj_2}+\cdots+\go^{aj_d}: a\in\ZZ_p, a\ne 0 \} $$
%are disjoint for unitarily inequivalent frames.

For $n=p$ prime we are able to give an 
explicit formula for the count (\ref{mndform}) of
$m_{p,d}=h_{p,d}$.
It is convenient to split this into the {\it lifted} and 
{\it unlifted} harmonic frames. We say that a harmonic frame 
given by $J\subset\ZZ_n$ is {\bf lifted} if $0\in J$,
equivalently, the subset of characters defining it contains
the trivial character $1\in\hat\ZZ_n$, 
or its vectors have a nonzero sum.

\begin{theorem} 
\label{primeframe}
Let $p$ be a prime, and $h_{p,d}^{\rm u}$
and $h_{p,d}^{\rm l}$ be the number of unlifted and lifted harmonic
frames of $p$ distinct vectors for $\Cd$ up to unitary equivalence.
For %Then $h_{p,1}^{\rm u}=1$, $h_{p,1}^{\rm l}=0$,
$d>1$,
we have
\begin{align}
h_{p,d}^{\rm u} 
& = {1\over p-1}\sum_{j|\gcd(p-1,d)}{{p-1\over j}\choose{d\over j}}\varphi(j), 
\label{primeframeI}
\\
h_{p,d}^{\rm l} & = {1\over p-1}\sum_{j|\gcd(p-1,d-1)}{{p-1\over j}\choose{d-1\over j}}\varphi(j).
\label{primeframeII}
\end{align}
\end{theorem}

%\begin{theorem} Let $p$ be a prime, and $N_p^{\rm u}=N_p^{\rm u}(d)$,
%$N_p^{\rm l}=N_p^{\rm l}(d)$ be the numbers of unlifted and lifted harmonic
%frames of $p\ge d$ vectors for $\Cd$, $d>1$. Then
%\begin{align*}
%N_p^{\rm u} & = {1\over p-1}\sum_{j|\gcd(p-1,d)}{{p-1\over j}\choose{d\over j}}\varphi(j), \\
%N_p^{\rm l} & = {1\over p-1}\sum_{j|\gcd(p-1,d-1)}{{p-1\over j}\choose{d-1\over j}}\varphi(j).
%\end{align*}
%\end{theorem}

\begin{proof} For $d=1$, the unique harmonic frame (with distinct
vectors) is unlifted, so that $h_{p,1}^{\rm u}=1$, $h_{p,1}^{\rm l}=0$.
We observe that the formula (\ref{primeframeI}) also holds for $d=1$.
For $d>1$, all $d$--element subsets of
$\ZZ_p$ generate $\ZZ_p$, 
and so all harmonic frames have distinct vectors.

We first count the unitarily inequivalent unlifted 
harmonic frames, i.e., the number of $d$--element subsets
of $\ZZ_p\setminus \{ 0 \}$ up to multiplicative equivalence.
If $a \in \ZZ_p^*$ has order $j$, then its action on $\ZZ_p \setminus \{ 0 \}$ 
gives $\tfrac{p-1}{j}$ orbits of size $j$.  
In order for there to be a $d$--element subset $J$ of $\ZZ_p \setminus\{0\}$ 
fixed by $a$, we must have $j | d$, 
and the number of such subsets is $|\Fix(a)|={{p-1\over j}\choose{d\over j}}$.
There are $\varphi(j)$ elements in $\ZZ_p^*$ of order $j$, and so
Burnside's theorem (\ref{Burnsidect}) applied to $S$ the collection
of $d$--element subsets of $\ZZ_p\setminus\{0\}$ gives
the first formula:
$$ h_{p,d}^{\rm u}
= {1\over|\ZZ_p^*|} \sum_{j|\gcd(p-1,d)} \sum_{a\in\ZZ_p^* \atop\ord(a)=j} |\Fix(a,S)|
= {1\over p-1} \sum_{j|\gcd(p-1,d)} {{p-1\over j}\choose{d\over j}}\varphi(j).  $$
%We observe that this formula also holds for $d=1$.

We now count the lifted frames. 
These are given by the $d$--element subsets $J\subset\ZZ_p$ 
with $0\in J$, which are multiplicatively equivalent if and only if 
the $(d-1)$--element subsets $J \setminus\{0\}$ are. 
Thus $h_{p,d}^{\rm l}=h_{p,d-1}^{\rm u}$,
which gives (\ref{primeframeII}) since the formula for $h_{p,d}^{\rm u}$ holds for $d \ge 1$.
%Since all nontrivial characters have distinct entries,
%all harmonic frames for $d>1$ will have distinct vectors.
% We first count the multiplicatively inequivalent unlifted frames,
%i.e., those obtained by choosing $d$ rows of the character table not
%including the first row, or, equivalently a $d$--element subset of
%$\ZZ_p \setminus\{0\}$.
%
%If $a \in \ZZ_p^*$ has order $j$, then $a$ acts on $\ZZ_p \setminus \{ 0 \}$ with $\tfrac{p-1}{j}$ orbits of size $j$.  In order for there to exist a $d$ element subset of $\ZZ_p \setminus \{ 0 \}$ fixed by $a$, we must have $j | d$, and then the number of such subsets is $|\Fix(a)|={{p-1\over j}\choose{d\over j}}$.  There are $\varphi(j)$ elements in $\ZZ_p^*$ of order $j$, and summing over $j$ gives the desired formula.
%When counting lifted frames, one choice of character is determined for us.  
%The number of possibilities for the remaining choices may be obtained by applying the argument above to $d-1$ instead of $d$.
\end{proof}

A backwards recursive formula for $h_{p,d}$ based on orbit counting
is given in \cite{H10}.

\begin{example} For $d=2$ and $p > 2$, we have $h_{p,2}={1\over2}(p+1)$, since
$$ h_{p,1}^{\rm u}=h_{p,2}^{\rm l}=1, \qquad
h_{p,2}^{\rm u}=h_{p,3}^{\rm l}
={1\over p-1}\Bigl\{ {p-1\choose2}+{{p-1\over2}\choose 1}\Bigr\}
={1\over2}(p-1). $$
\end{example}

\begin{example} For $d=3$ and $p > 2$, we have
$$ h_{p,3}^{\rm u} = h_{p,4}^{\rm l} 
={1\over p-1}
\begin{cases}
{p-1\choose 3}, & p\not\equiv 1 \pmod 3; \cr
{p-1\choose 3}+2{{p-1\over 3}\choose 1}, & p\equiv 1 \pmod 3.
\end{cases} $$
Hence
$$ h_{p,3} = 
\begin{cases}
{1\over 6}(p^2-2p+3), & p\not\equiv 1 \pmod 3; \cr
{1\over 6}(p^2-2p+7), & p\equiv 1 \pmod 3. \cr
\end{cases} $$
The above formulas for $ p\equiv 1 \pmod 3$ and $ p\equiv 2 \pmod 3$
appear in \cite{H10} (Prop.\ 4.2).
\end{example}

\begin{example}
For $d=4$ and $p > 2$, we have
$$ h_{p,4}^{\rm u} = h_{p,5}^{\rm l} 
={1\over p-1}
\begin{cases}
{p-1\choose 4}+{{p-1\over2}\choose2}, & p\not\equiv 1 \pmod 4; \cr
{p-1\choose 4}+{{p-1\over2}\choose2}+2{{p-1\over4}\choose1},
& p\equiv 1 \pmod 4.
\end{cases} $$
As indicated, we can construct formulas for $h_{p,d}$ depending
on $p$ modulo $d$ and $d-1$, e.g.,
$$ h_{p,4}=h_{p,4}^{\rm u}+h_{p,4}^{\rm l}
= {1\over 24}(p^3-5p^2+9p+19),
\qquad p\equiv 1\pmod{12}. $$
\end{example}

It is also possible to count the number of {\it real} harmonic frames.
We recall that  $J\subset\ZZ_n$ gives a real harmonic frame
if and only if it is closed under taking inverses, i.e., $J=-J$. 
For $n=p$ an odd prime, $-j=j$ if and only if $j=0$, and so the $J$
giving real frames have $0 \notin J$ when $d$ is even,
and $0 \in J$ when $d$ is odd. 
Burnside counting gives the following.

\begin{proposition} Let $p$ be an odd prime and $d>1$. For $d$ even, 
the number of real harmonic (unlifted) frames of $p$ distinct vectors for $\Rd$
(up to unitary equivalence) is 
$$ h_{p,d}^\RR = {1\over p-1} \Bigl\{ \sum_{j|\gcd(p-1,d)\atop j\ {\rm even}}
{{p-1\over j}\choose{d\over j}}\varphi(j) 
+\sum_{j|\gcd(p-1,{d\over 2})\atop j\ {\rm odd}}
{{p-1\over 2 j}\choose{d\over2 j}}\varphi(j) \Bigr\} . $$
%$$ h_{p,d}^\RR = {2\over p-1}
%\sum_{j|\gcd({p-1\over2},{d\over2})}{{p-1\over 2j}\choose{d\over 2j}}\varphi(j)
%\approx p^{{d\over2}-1}. $$ 
For $d$ odd, 
the number of real harmonic (lifted) frames of $p$ distinct vectors for $\Rd$ is 
$$ h_{p,d}^\RR = {1\over p-1} \Bigl\{ \sum_{j|\gcd(p-1,d-1)\atop j\ {\rm even}}
{{p-1\over j}\choose{d-1\over j}}\varphi(j) 
+\sum_{j|\gcd(p-1,{d-1\over 2})\atop j\ {\rm odd}}
{{p-1\over 2 j}\choose{d-1\over2 j}}\varphi(j) \Bigr\}. $$
%$$ h_{p,d}^\RR 
%= {2\over p-1}
%\sum_{j|\gcd({p-1\over2},{d-1\over2})}{{p-1\over 2j}\choose{d-1\over 2j}}\varphi(j)
%\approx p^{{d-3\over2}}. $$
\end{proposition}

\begin{proof}
We first consider the case when $d$ is even.  The unit group $\ZZ_p^*$ is cyclic of even order $p-1$, and we let $a\in\ZZ_p^*$ have order $j$.  We wish to count the number of $d$-element sets $J \subset \ZZ_p \setminus \{ 0 \}$ that are invariant under multiplication by $a$ and $-1$.  If $j$ is even, then $-1=a^{j\over 2}$, and so this is equal to the number of subsets invariant under multiplication by $a$.  This is ${p-1\over j}\choose{d\over j}$ as in the proof of Theorem \ref{primeframe}.  If $j$ is odd, then the subgroup of $\ZZ_p^*$ generated by $-1$ and $a$ is cyclic with generator $-a$.  We therefore wish to find the number of subsets invariant under multiplication by $-a$, and as this element has order $2j$, this is ${p-1\over 2 j}\choose{d\over2 j}$.

%, and so has 
%an element with even order. % elements with even and odd orders. 
%The action of $a$ on $\ZZ_p\setminus\{0\}$ gives ${p-1\over j}$ orbits of size $j$.
%any nonzero element of $\ZZ_p$ gives an orbit of size $j$.

%Suppose that $a$ fixes some $J=K\cup -K$, $|J|=d$.
%If $j$ is even, then $-1=a^{j\over 2}$, 
%so $J$ consists of ${d\over j}$ orbits under the action of $a$
%(half the orbit in $K$ and the other half in $-K$), and we must have $j|d$. 
%If $j$ is odd, then $K$ consists of ${d/2\over j}$ orbits
%of size $j$ (their negatives give the remaining ${d/2\over j}$ orbits 
%which make up $J$), and we must have $j|(d/2)$

Thus Burnside
orbit counting gives
$$ h_{p,d}^\RR = {1\over p-1} \Bigl\{ \sum_{j|\gcd(p-1,d)\atop j\ {\rm even}}
{{p-1\over j}\choose{d\over j}}\varphi(j) 
+\sum_{j|\gcd(p-1,{d\over 2})\atop j\ {\rm odd}}
{{p-1\over 2 j}\choose{d\over2 j}}\varphi(j) \Bigr\} . $$

When $d$ is odd, the subsets $J$ giving real frames are multiplicatively
equivalent if and only if the sets $J \setminus \{ 0 \}$ are, and so we
may apply the previous count (with $d$ replaced by $d-1$).
%The formula so obtained also holds for $d=1$.
\end{proof}

\begin{example}
For $d = 2,3$, there is a single real harmonic frame of $p$ distinct vectors, i.e.,
$$ h_{p,2}^\RR =h_{p,3}^\RR =1. $$
For $d$ even, $d\ge4 $, we have the estimate
$$  h_{p,d}^\RR = h_{p,d+1}^\RR \approx p^{{d\over2}-1}, \qquad
p\to\infty. $$
\end{example}

%For $n=p$ prime, $d$ even, the number of (unlifted) real harmonic frames is
%$$ N = {2\over p-1}
%\sum_{j|\gcd({p-1\over2},{d\over2})}{{p-1\over 2j}\choose{d\over 2j}}\varphi(j)
%\approx p^{{d\over2}-1}, $$
%and for $n=p$ prime, $d$ odd, the number of (lifted) real harmonic frames is
%$$ N = {2\over p-1}
%\sum_{j|\gcd({p-1\over2},{d-1\over2})}{{p-1\over 2j}\choose{d-1\over 2j}}\varphi(j)
%\approx p^{{d-3\over2}}, $$

%\begin{example}
%The first case of cyclic frames with the same multiset of inner products,
%giving the same frame is $n=8$, $d=3$. Take the 8 and 9 th rows
%$$ \{ 1, 2, 5 \}, \{ 3, 6, 7 \}, \qquad \{ 1, 5, 6 \}, \{ 2, 3, 7 \} $$
%These frames are not equivalent via an isomorphism.
%\end{example}

%\begin{example}
%The first case of cyclic frames with the same multiset of inner products,
%but giving different frames is $n=12$, $d=3$. Take rows 68 and 70, the
%$\ZZ_{12}$ multiplicative equivalence classes are
%$$ \{ 1, 7, 11 \}, \{ 5, 7, 11 \}, \{ 1, 5, 7 \}, \{ 1, 5, 11 \}, \qquad
%\{ 1, 3, 9 \}, \{ 3, 5, 9 \}, \{ 3, 7, 9 \}, \{ 3, 9, 11 \} $$
%\end{example}

%\begin{example} For $n=9$, $d=4$ take the 2 and 3 frames:
%$$\{  2, 3, 5, 8 \}, \{ 1, 4, 6, 7 \}, \qquad
%\{ 1, 3, 4, 7 \}, \{ 2, 5, 6, 8 \} $$
%These are not multiplicatively equivalent, and $\go^{j_1}+\cdots+\go^{j_d}$
%is fixed by a subgroup of the Galois group, but are
%\end{example}

\section{Projective unitary equivalence of harmonic frames}
%\section{Heuristics for the number of projective equivalence classes}

Many applications of tight frames $(v_j)$ are based on the expansion
(\ref{redundonexp}), i.e., depend only on the vectors up to
unit scalar multiples. We say that tight frames $(v_j)$ and $(w_k)$
are {\bf projectively unitarily equivalent} (up to a reindexing) if
there is a unitary map $U$, unit modulus scalars $c_j$,
and a bijection $\gs:j\to k$ (a reindexing)
between their index sets for which 
$$ v_j = c_j Uw_{\gs j}, \qquad \forall j. $$
In \cite{CW16} it is shown that harmonic frames given by subsets
$J,K\subset G$ are projectively unitarily equivalent (with $\gs$ the
identity) if $J$ and $K$ are translates, i.e., $K=J-b$, % for some 
$b\in G$.
Therefore the {\it affine transformations} $L_{(\gs,b)}$ given by
$$ L_{(\gs,b)} g:= \gs g +b, \qquad \gs\in\Aut(G), \ b\in G $$
map subsets $J\subset G$ to subsets which give projectively 
unitarily equivalent harmonic frames (via an automorphism). 
Calculations of \cite{CW16} suggest that the majority of 
projective unitary equivalences occur in this way, 
via a reindexing which is an automorphism (indeed there is no 
known case where it does not). We observe that every $d$--element 
subset of $\ZZ_n$ is a translate of one which generates $\ZZ_n$,
and so every cyclic harmonic frame is projectively similar
to one with distinct vectors.

We now count the number $p_{n,d}$
of cyclic harmonic frames $\Phi_J$  for  $\Cd$
up to this projective unitary equivalence
via an affine transformation of the index set $J$.
Since the affine group (group of affine transformations) has order $n\varphi(n)$ and there
are ${n\choose d}$ subsets of $\ZZ_n$ of size $d$, we have
$$ p_{n,d} \ge {\binom{n}{d}\over n\varphi(n)} 
\gg {n^{d-1}\over\varphi(n)}\ge n^{d-2}, \qquad n\to\infty. $$
For $d\ge4$, we can establish this rate of growth:

\begin{theorem}
Let $p_{n,d}$ be the number of orbits of the affine group acting
on the $d$--element subsets of $\ZZ_n$. % which generate $\ZZ_n$. 
For $d\ge4$, we have 
\begin{equation}
\label{projequivest}
p_{n,d} \approx {n^{d-1}\over \varphi(n)} \ge n^{d-2}, \qquad n\to\infty.
\end{equation}
\end{theorem}

%\begin{proposition}
%Let $d \ge 4$.  If $p_n$ is the number of orbits of the affine group on $d$-element subsets of $\ZZ_n$, we have $p_n \sim n^{d-1} / \varphi(n)$.
%\end{proposition}

\begin{proof}
Since $\ZZ_n^*$ gives the automorphisms of $\ZZ_n$,
the group of affine transformations of $\ZZ_n$ 
is isomorphic to %the semidirect product 
$\ZZ_n^* \ltimes \ZZ_n$, with 
$(a,b) \in \ZZ_n^* \ltimes \ZZ_n$ acting on $\ZZ_n$ via
$x \mapsto ax + b$.

%Since the affine group has order $n\varphi(n)$ and there
%are ${n\choose d}$ subsets of $\ZZ_n$ of size $d$, we have
%$$ p_n \ge \binom{n}{d} / n\varphi(n) \gg n^{d-1} / \varphi(n). $$
To obtain an upper bound for $p_{n,d}$,
we estimate the terms in the Burnside orbit counting formula
\begin{equation}
\label{fixsumform}
p_{n,d} = {1\over n\varphi(n)} 
\sum_{ (a,b) \in \ZZ_n^* \ltimes \ZZ_n } |\Fix(a,b)|,
\end{equation}
where $\Fix(a,b)$ %=\Fix\bigl((a,b)\bigr)=\Fix((a,b))$ 
is the collection of $d$--element
subsets $A\subset\ZZ_n$ fixed by the action of $(a,b)$.
The orbit of $x$ under the action of $(a,b) \in \ZZ_n^* \ltimes \ZZ_n$ is
$$%\{\quad 
x, \quad ax+b, \quad
a^2 x+ ab+b, \quad\ldots. $$ % \quad a^3x+a^2b+ab+b, 
%\ a^jx+(a^{j-1}+\cdots+a+1)b, \ \ldots . $$
Since $ax+b=x+((a-1)x +b )$, all orbits %the orbit of $x$ 
will have at least two elements provided that $b\not\in (a-1)\ZZ_n$. 
Thus, our assumption $d \ge 4$ implies that any $A$ fixed by $(a,b)$ with $b\not\in (a-1)\ZZ_n$ will have
at least two of its elements determined by the fact it %that $A$ 
is a union of orbits. This implies the contribution to the sum in (\ref{fixsumform})
by these elements is at most $n^{d-2} | \ZZ_n^* \ltimes \ZZ_n | \le n^d$.
It therefore remains to show the contribution to the sum in (\ref{fixsumform})
from the elements $(a,b)$ with $b\in(a-1)\ZZ_n$ is $\ll n^d$.

Suppose that $b\in(a-1)\ZZ_n$.
Conjugating $(a,b)$ by $(1,c)$ 
does not change the size of
$\Fix(a,b)$. % its fixed point set, 
Since $(1,c) (a,b) (1,c)^{-1} = (a, b + c(1-a) )$ and $b \in (a-1) \ZZ_n$, 
we may choose a $c$ so that $(1,c) (a,b) (1,c)^{-1} = (a,0)$.  
Thus $|\Fix(a,b)|=|\Fix(a,0)|$.
We observe the action of $(a,0)$ and $a$ on $\ZZ_n$ is the same.
 %and recall that
 %the number of $d$--element subsets which generate $\ZZ_n$ 
 %and are fixed by $a$ was estimated in Lemma \ref{fixbound}.
%where $\Fix(a,0)$ is the collection of $d$--element subsets 
%$A\subset\ZZ_n$ fixed by the action of $a\in\ZZ_n^*$
%(the number of such subsets % $A$ 
%which generate $\ZZ_n$ was estimated in Lemma \ref{fixbound}).
%It therefore suffices to count fixed points of $(a,0)$.
%Let $\Fix(a)$ be the collection of $d$--element subsets of $\ZZ_n$ fixed by $a\in\ZZ_n^*$.
Let $m := \text{gcd}(a-1,n)$. Then 
$$ (a-1)\ZZ_n = m\ZZ_n,$$
and the % The 
subgroup of $\ZZ_n$ on which $a$ acts trivially is
$$ H=\{x\in\ZZ_n:ax=x\} =\hbox{${n\over m} \ZZ_n$}. $$
%\qquad m := \text{gcd}(a-1,n). $$
We partition $\Fix(a,0)=\cup_j F_j$, where each $A\in F_j$ has
exactly $j$ elements not in $H$, i.e.,
$$ F_j:=\{A\in\Fix(a,0):|A\setminus H|=j\}, \qquad
j=0,1,\ldots,d. $$
If $x\notin H$, 
then $ax\ne x$ and $ax\not\in H$ 
(otherwise $ax = a^{-1}a(ax)=a^{-1}ax=x$), so that 
%$|F_1|=0$, and 
%$$ |F_j|\le |H|^{d-2} n = m^{d-2}n, \qquad j\ge 2. $$
$$ |F_0|\le|H|^d=m^d, \qquad
|F_1|=0, \qquad 
|F_2|\le |H|^{d-2} n=m^{d-2}n, \qquad
\sum_{j\ge 3}|F_j|\le n^{d-2}. $$ % = m^{d-2}n, \quad j\ge 2. $$
The last inequality holds because any $A$ with at least 3 elements not in $H$ has at least 2 elements determined by the fact it is a union of orbits.  Using these, we have the estimate
\begin{align*}
\sum_{ (a,b) \in \ZZ_n^* \ltimes \ZZ_n
\atop b\in(a-1)\ZZ_n } |\Fix(a,b)|
& \le \sum_{m|n}\sum_{a\in\ZZ_n^*\atop m = \gcd(a-1,n)}
\sum_{b\in m\ZZ_n} |\Fix(a,0)|
\le \sum_{m|n}  {n\over m} {n\over m}
\Bigl(|F_0|+|F_2|+\sum_{j\ge3} |F_j|\Bigr)\cr
& \le n^d \sum_{m|n} \left({n\over m}\right)^{2-d} 
+ n^{d-1} \sum_{m|n} \left({n\over m}\right)^{4-d} 
+  n^d \sum_{m|n} {1\over m^2} \cr
&  \le n^d \sum_k {1\over k^2}
+ n^{d-1} \sum_{m|n} 1
+  n^d \sum_{m} {1\over m^2}
\ll n^d,
\end{align*}
which completes the proof.
\end{proof}

\begin{lemma}
For $n$ prime, (\ref{projequivest}) also holds for $d=3$.
\end{lemma}

\begin{proof}  If $n$ is prime, then for any 3-element subset $A \subset \ZZ_n$ there is an affine transformation $(a,b)$ so that $(a,b)A$ contains 0 and 1.  There are at most $n$ choices for the third element, so that $n \ge p_{n,3}$.

\end{proof}

\noindent
{\bf\large Acknowledgements}
\medskip

The first author was supported by NSF grants DMS-1509331 and DMS-1501230, and would like to thank the University of Auckland for its hospitality while part of this work was carried out.

%\section{Group theoretic estimates}

%\bibliography{penguin}

% References

%\bibliographystyle{abbrv}
\bibliographystyle{alpha}
\nocite{*}
\bibliography{references}

\end{document}